\documentclass[a4paper,12pt]{amsart}

	\usepackage{amsmath,amsthm,amsfonts,amssymb,mathrsfs}
   \usepackage{graphicx}
   \usepackage{color}
   \usepackage[nice]{nicefrac}
	\usepackage{dsfont}
	\usepackage{delarray}
	\usepackage{enumerate}

   \newtheorem{theo}{Theorem}[section]
	
   \newtheorem{lemm}[theo]{Lemma}
   \newtheorem{coro}[theo]{Corolary}

	\theoremstyle{remark}
	\newtheorem{rema}[theo]{Remark}

   \def\dowod{\par\noindent{\sc Proof.}~}
   \def\cbdu{\hfill{$\Box$}}

   \usepackage{lineno}
   \usepackage{verbatim}
   \usepackage{fancyhdr}

   \def \sgn{\text{sgn\,}}

   \def \div{ \nabla \cdot}
   
   \def \Z{\mathbb Z}
   
   \newcommand{\dx}{\,\text dx}
   \newcommand{\dy}[1]{\,\text d#1}

	\newcommand\R{\mathbb{R}}

\numberwithin{equation}{section}
\numberwithin{theo}{section}

\begin{document}

\title{Asymptotics behaviour in one dimensional model of interacting particles}

\author{Rafa\l\ Celi\'nski}
\address{
 Instytut Matematyczny, Uniwersytet Wroc\l awski,
 pl. Grunwaldzki 2/4, 50-384 Wroc\-\l aw, POLAND}
\email{Rafal.Celinski@math.uni.wroc.pl}

\date{\today}

\thanks{
Author was supported by the International Ph.D.
Projects Programme of Foundation for Polish Science operated within the
Innovative Economy Operational Programme 2007-2013 funded by European
Regional Development Fund (Ph.D. Programme: Mathematical
Methods in Natural Sciences) and by the MNiSzW grant No. N N201 418839. This is a part of the author Phd dissertation written under supervision of Grzegorz Karch.
}

\begin{abstract} 
We consider the equation $u_t=\varepsilon u_{xx}+(u\ K'*u)_x$ for  $x\in\mathbb{R}$, $t>0$ and with $\varepsilon\geq 0$, supplemented with a nonnegative, integrable initial datum. We present a class of interaction kernels $K^\prime$ such that the large time behaviour of solutions to this initial value problem is described by a compactly supported self-similar profile.
\end{abstract}

\bigskip


\maketitle


\section{Introduction}

We study the asymptotic behaviour of solutions to the one-dimensional initial value problem
\begin{align}
&u_t = \varepsilon u_{xx} + \left(u\ K^\prime \ast u\right)_{x} \quad \text{for}\ x
 \in \R,\ t > 0, \label{aggr}\\
&u(x, 0) = u_0 (x)\quad \text{for}\ x \in \R, \label{aggr-ini}
\end{align}
where the {\it interaction} kernel $K^\prime$ is a given function, an initial datum $u_0 \in L^1 (\R)$ is nonnegative and $\varepsilon\geq 0$.

Equation \eqref{aggr} arises in study of an animal aggregation as well as in some problems in mechanics of continuous media. The unknown function $u=u(x,t)$ represents either the population density of a species or, in the case of materials applications, a particle density. The kernel $K^\prime$ in \eqref{aggr} can be understood as the derivative of a certain function $K$, that is, $K^\prime$ stands for $\dy{K}/\dx$. We use this notation to emphasise that the cell interaction described by equation \eqref{aggr} takes place by means of a potential $K$. Moreover, our assumptions on interaction kernel $K^\prime$ imply that equation \eqref{aggr} describe particles interacting according to a repulsive force (this will be clarified bellow).

Let us first notice that the one-dimensional parabolic-elliptic system of chemotaxis
\begin{align}\label{chemo}
u_t=\varepsilon u_{xx}-(uv_x)_x,\qquad -v_{xx}+v=u, \qquad x\in\mathbb{R},\ \ t>0
\end{align}
can be written as equation \eqref{aggr}. Indeed, if we put $K(x)=-\frac{1}{2}e^{-|x|}$ into the \eqref{aggr}, which is the fundamental solution of the operator $\partial^2_x-$Id, one can rewrite the second equation of \eqref{chemo} as $v=-K*u$. Here, however, we should emphasise that, below we consider repulsive phenomena, where the interaction kernel has the opposite sign, see Remark \ref{remark2} for more details.

This work is motivated by the recent publication by Karch and Suzuki \cite{KS10} where the authors study the large time asymptotics of solutions to \eqref{aggr}-\eqref{aggr-ini} under the assumption $K^\prime\in L^1(\R)$. They showed that either the fundamental solution of the heat equation or a nonlinear diffusion wave appear in the asymptotic expansion of solutions as $t\to\infty$. Analogous results on the solutions to the one dimensional chemotaxis model \eqref{chemo} can be found in \cite{NSU03,NY07}. Here, we would like to point out that, in all those results, a diffusion phenomena play a pivotal role in the large time behaviour of solutions to problem \eqref{aggr}-\eqref{aggr-ini}.

The main goal of this work is to show that for a large class of interaction kernels $K^\prime\in L^\infty(\R)\backslash L^1(\R)$, the diffusion is completely negligible in the study of the large time asymptotics of solutions. Let us be more precise. Our assumption on the interaction kernel imply that $K^\prime(x)$ is sufficiently small perturbation of the function $-\frac{A}{2}H(x)$, where, $A\in(0,\infty)$ is a constant and $H$ is the classical sign function given by the formula: $H(x)=-1$ for $x<0$ and $H(x)=1$ for $x>0$ ({\it cf.} Remark \ref{remark1}). Under these assumptions, we show that for large values of time, a solution of problem \eqref{aggr}-\eqref{aggr-ini} looks as a compactly supported self-similar profile, defined as the space derivative of a rarefaction wave, {\it i.e.} the solution of the Riemann problem for the nonviscous Burgers equation $u_t+Auu_x=0$ (see Corollary \ref{corolary} for the precise statement).

In our reasoning, first, we consider $\varepsilon>0$, and our result on the large time behaviour are, in some sense, independent of $\varepsilon$. Next, we pass to the limit $\varepsilon\to 0$ to obtain an analogous result for the inviscid aggregation equation $u_t-(u\ K^\prime*u)_x=0$. In particular, our assumptions imply that weak, nonnegative solutions to the initial value problem for this inviscid equation exists for all $t>0$.

To conclude this introduction, we would like to recall, that the multidimensional inviscid aggregation equation $u_t-\div(u\nabla K*u)=0$ was derived as a macroscopic equation from the so-called ``individual cell-based mode'' \cite{BV05,S00}, namely, as a continuum limit for a system of particles $X_k(t)$ placed at the point $k$ in time $t$ and evolving by the system of differential equations:
\begin{align*}
\frac{dX_k(t)}{dt}=-\sum_{i\in\Z\backslash\{k\}}\nabla K(X_k(t)-X_i(t)),\quad k\in\Z
\end{align*}
where $K$ is the potential. Results on the local and global existence as well as the blow-up of solutions of this inviscid aggregation equation one can find in \cite{BCL09,BH10,BL07,L07} and in references therein.

\subsection*{Notation}
In this work, the usual norm of the Lebesgue space $L^p (\R)$ with respect to the spatial variable is denoted by $\|\cdot\|_p$ for any $p \in [1,\infty]$ and $W^{k,p}(\R)$ is the corresponding Sobolev space. The set $C^\infty_c (\R)$ consist of smooth and compactly supported functions. Moreover $(f*g)(x)$ denotes the usual convolution, {\it i.e.} $(f*g)(x)=\int_\R f(x-y)g(y)\dy{y}$. The letter $C$ corresponds to a generic constants (always independent of $x$ and $t$) which may vary from line to line. Sometimes, we write,  {\it e.g.}  $C=C(\alpha,\beta,\gamma, ...)$ when we want to emphasise the dependence of $C$ on parameters~$\alpha,\beta,\gamma, ...$.

\section{Main results} \label{sec2}

We begin our study of large time behaviour of solution by recalling that, for $\varepsilon>0$, the initial value problem \eqref{aggr}-\eqref{aggr-ini} is known to have a unique and global-in-time solution for a large class of initial conditions $u_0$ and interaction kernels $K^\prime$. Such results are more-or-less standard and the detailed reasoning can be found in \cite{KS09}. In particular, our assumptions (see Theorem \ref{thm_lp_decay} below) imply that $K^\prime\in L^\infty(\R)$, hence the kernel $K^\prime$ is mildly singular in the sense stated in \cite[Thm 2.5]{KS09}. In this case, results from \cite{KS09} can be summarised as follows: for every $u_0\in L^1(\R)$ such that $u_0\geq 0$, there exists the unique global-in-time solution $u$ of problem \eqref{aggr}-\eqref{aggr-ini} satisfying
\begin{align*}
u\in C\left([0,+\infty),\ L^1(\R)\right)\cap C\left((0,+\infty),\ W^{1,1}(\R)\right)\cap C^1\left((0,+\infty),\ L^1(\R)\right).
\end{align*}
In addition, the condition $u_0(x)\geq 0$ implies $u(x,t)\geq 0$ for all $x\in\R$ and $t\geq 0$. Moreover we obtain the conservation of the $L^1$-norm of nonnegative solutions:
\begin{align}\label{mass}
\|u(t)\|_{L^1}=\int_\R u(x,t)\dx = \int_\R u_0(x)\dx = \|u_0\|_{L^1}.
\end{align}
In Theorem \ref{thm_eps} below, we pass to the limit $\varepsilon\to 0$, to obtain nonnegative weak solutions of problem \eqref{aggr}-\eqref{aggr-ini} with $\varepsilon=0$, for which the conservation of mass \eqref{mass} holds true, as well.

The goal of this work is to study the large time behaviour of solution to \eqref{aggr}-\eqref{aggr-ini}. First, we state conditions under which these solutions decay as $t\to\infty$.

\begin{theo}[Decays of $L^p$ norm]\label{thm_lp_decay}
Assume that $u=u(x,t)$ is a nonnegative solution to problem \eqref{aggr}-\eqref{aggr-ini} with $\varepsilon>0$, where the interaction kernel has the form $K^\prime(x)=-\frac{A}{2}H(x)+V(x)$, where $H$ is the sign function, $A>0$ is a constant, and the function $V$ satisfies
\begin{align}
&V\in W^{1,1}(\R) \text{ with } \|V_x\|_{L^1}<A\label{K_2}.
\end{align}
Suppose also that $u_0 \in L^1 (\R)$ is nonnegative. Then for every $p\in[1,\infty]$ the following inequality hold true
\begin{equation}\label{eq_lp_decay}
||u(t)||_p\leq \left(A-||V_x||_1\right)^\frac{1-p}{p} ||u_0||_1^{1/p}\ t^\frac{1-p}{p}
\end{equation}
for all $t>0$.
\end{theo}

\begin{rema}\label{remark1}
Notice that, under assumption \eqref{K_2}, we have $V(x)=\int_{-\infty}^xV_y(y)\dy{y}$. Hence, we get immediately that $V\in L^\infty(\R)\cap C(\R)$, $\lim_{|x|\to\infty}V(x)=0$, and the following estimate, $\|V\|_\infty\leq\|V_x\|_1<A$ hold true. Consequently, our assumption on the interaction kernel $K^\prime$ imply that $K^\prime+\frac{A}{2} H\in C_0(\R)$ (continuous and decaying at infinity functions). This means that the kernel $K^\prime$ has to jump at zero exactly as the rescaled sign function $-\frac{A}{2}H$ and has to converge to the constants $\pm\frac{A}{2}$ as $x\to\mp\infty$, respectively. In some sense, this means that the potential $K(x)$ looks as $-\frac{A}{2}|x|$ at $x=0$ and as $|x|\to+\infty$.
\end{rema}

\begin{rema}\label{remark2}
Our assumptions on the kernel $K^\prime(x)$ imply that interactions between particles are similar as in the chemorepulsion motion in chemotaxis phenomena, namely, when regions of high chemical concentrations have a repulsive effect on particles. Such a model was studied for example in \cite{CLM08}.
\end{rema}

In the next step of this work, we derive an asymptotic profile as $t\to\infty$ of solutions \eqref{aggr}-\eqref{aggr-ini}. First, notice that if the large time behaviour of a solution to problem \eqref{aggr}-\eqref{aggr-ini} is described by the heat kernel or the nonlinear diffusion wave (as {\it e.g.} in \cite{KS10}) then we expect the following decay rate $||u(t)||_p\leq C\ t^\frac {1-p}{2p}$ for all $t>0$. Observe, that the function $u$ from Theorem \ref{thm_lp_decay} decays faster, hence, its asymptotic behaviour as $t\to\infty$ should be different.

From now on, without loss of generality, we assume that $\int_\R u(x,t)\dx = \int_\R u_0(x)\dx=1$. Indeed, due to the conservation of mass \eqref{mass}, it suffices to replace $u$ in equation \eqref{aggr} by $\frac{u}{\int_\R u_0\dx}$ and $K^\prime$ by $K^\prime\int_\R u_0\dx$. 

Next, let us put 
\begin{equation}\label{primit}
U(x,t)=\int_{-\infty}^x u(y,t)\dy{y} -\frac{1}{2},
\end{equation}
where $u(x,t)$ is the solution of \eqref{aggr}-\eqref{aggr-ini}. Since $u=U_x$, using the explicit form of the kernel $K^\prime$ (cf. Lemma \ref{conv_ker} below), we obtain that the primitive $U=U(x,t)$ satisfy the following equation
\begin{align}\label{primitive}
U_t=\varepsilon U_{xx}-AUU_x+U_x\ V*U_x,
\end{align}
which can also be considered as a nonlinear and nonlocal perturbation of the viscous Burgers equation.

Our main result says that the large time behaviour of $U$ is described by a self-similar profile, given by a rarefaction wave, namely, the unique entropy solution of the Riemann problem for the scalar conservation law
\begin{align}
W_t^R+AW^RW^R_x &=0 \label{Riemann}\\
W^R(x,0) &=\frac{1}{2}H(x) \label{Riemann_ini}.
\end{align}
\noindent It is well-known (see {\it e.g.} \cite{evans}) that this rarefaction wave is given by the explicit formula 

\begin{equation}\label{raref}
W^R(x,t):=\left\{\begin{aligned}
-\frac{1}{2} &\qquad \textrm{for}\qquad x\leq-\frac{At}{2},\\
\frac{x}{At} &\qquad \textrm{for}\qquad -\frac{At}{2}<x<\frac{At}{2},\\
\frac{1}{2} &\qquad \textrm{for}\qquad x\geq\frac{At}{2}.
\end{aligned}
\right.
\end{equation}

\begin{theo}[Convergence towards rarefaction waves]\label{conv}
Let the assumptions of Theorem~\ref{thm_lp_decay} hold true. Assume, moreover, that a nonnegative initial datum $u_0(x)$ satisfies
\begin{align}\label{as-u0}
\int_\R u_0(x)\dx=1,\quad \text{and}\quad  \int_\R u_0(x)|x|\dx<\infty.
\end{align}
Then, there exist a constant $C>0$ independent of $\varepsilon$ such that for every $t>0$ and each $p\in(1,\infty]$ the following estimate hold true
\begin{equation}\label{conv_raref}
\|U(\cdot,t)-W^R(\cdot,t)\|_p\leq Ct^{-\frac{1}{2}\left(1-\frac{1}{p}\right)}\left(\log(2+t)\right)^{\frac{1}{2}(1+\frac{1}{p})},
\end{equation}
where $U=U(x,t)$ is the primitive of solution of problem \eqref{aggr}-\eqref{aggr-ini} given by \eqref{primit} and $W^R=W^R(x,t)$ is the rarefaction wave given by \eqref{raref}.
\end{theo}

Next, we show that the asymptotic formula \eqref{conv_raref} holds also true for weak solutions of problem \eqref{aggr}-\eqref{aggr-ini} with $\varepsilon=0$.

\begin{theo}\label{thm_eps}
Assume that the kernel $K^\prime$ has properties stated in Theorem \ref{thm_lp_decay} and the nonnegative initial condition $u_0\in L^1(\R)$ satisfies \eqref{as-u0}. Then the initial value problem
\begin{align}
U_t&=-AUU_x+U_xV*U_x \label{eq_U}\\
U(x,0)&=U_0(x)=\int_{-\infty}^xu_0(y)\dy{y}-\frac{1}{2} \label{U_ini}
\end{align}
has a weak solution $U\in C\big{(}\R\times (0,\infty)\big{)}$ such that $U_x\in L^\infty_{loc}\big{(}(0,\infty), L^\infty(\R)\big{)}$ that satisfies problem \eqref{eq_U}-\eqref{U_ini} in the following integral sense
\begin{equation*}
-\int_0^\infty \int_\R U\varphi_t\dx\dy{t} -\int_\R U_0(x)\varphi(x,0)\dx =\frac{A}{2}\int_0^\infty\int_\R U^2\varphi_x\dx\dy{t} + \int_0^\infty\int_\R U_x\,\big{(}V_x*U\big{)} \varphi\dx\dy{t}
\end{equation*}
for all $\varphi\in C^\infty_c(\R\times[0,+\infty))$.
This solution satisfies
\begin{equation}\label{thm_eps_WR}
\|U(\cdot,t)-W^R(\cdot,t)\|_p\leq Ct^{-\frac{1}{2}\left(1-\frac{1}{p}\right)}\left(\log(2+t)\right)^{\frac{1}{2}(1+\frac{1}{p})},
\end{equation}
for a constant $C>0$, for all $t>0$, and each $p\in(1,\infty]$.\end{theo}

Next, we use the result from Theorems \ref{conv} and \ref{thm_eps} to describe the large time asymptotics of solutions to problem \eqref{aggr}-\eqref{aggr-ini}.

\begin{coro}\label{corolary}
Let the assumptions either of Theorem \ref{thm_lp_decay} or Theorem \ref{thm_eps} hold true. For the solution $u=u(x,t)$ of problem \eqref{aggr}-\eqref{aggr-ini} with $\varepsilon\geq 0$ we define its rescaled version $u^\lambda(x,t)=\lambda u(\lambda x,\lambda t)$
for $\lambda>0$, $x\in\R$ and $t>0$. Then, for every test function $\varphi\in C^\infty_c(\R)$ and each $t_0>0$
\begin{align*}
\int_\R u^\lambda(x,t_0)\varphi(x)\dx\rightarrow -\int_\R W^R(x,t_0)\varphi_x(x)\dx \quad\text{as}\quad \lambda\to+\infty.
\end{align*}
\end{coro}

In other words, for each $t_0>0$, the family of rescaled solutions $u^\lambda(x,t_0)=\lambda u(\lambda x,\lambda t_0)$ to problem \eqref{aggr}-\eqref{aggr-ini} with $\varepsilon\geq 0$ converges weakly as $\lambda\to\infty$ to the compactly supported self-similar profile defined as 
\begin{equation}
\left(W^R\right)_x(x,t_0):=\left\{\begin{aligned}
\frac{1}{At} &\qquad \textrm{for}\qquad |x|<\frac{At}{2},\\
0 &\qquad \textrm{for}\qquad |x|\geq\frac{At}{2}.
\end{aligned}
\right.
\end{equation}

\section{Large time asymptotics}

In this section, we prove all results stated in Section \ref{sec2}. We begin by an elementary result.
\begin{lemm}\label{conv_ker}
Let $H$ be the sign function. For all $\varphi\in W^{1,1}(\R)$ the following inequality hold true: $H*\varphi_x=2\varphi.$
\end{lemm}
\dowod
First, we assume that $\varphi\in C^\infty_c(\R)$. Then
\begin{align*}
H*\varphi_x=\int_\R H(x-y)\varphi_y(y)\dy{y}=\int_{-\infty}^x\varphi_y(y)\dy{y}-\int_x^\infty \varphi_y(y)\dy{y}=2\varphi(x).
\end{align*}
The proof for general $\varphi\in W^{1,1}(\R)$ is completed by a standard approximation argument.

\cbdu

Now, we are in a position to prove Theorem \ref{thm_lp_decay} concerning the decay of solution in the $L^p$-spaces.
\begin{proof}[Proof of Theorem \ref{thm_lp_decay}]
Note, that, by \eqref{mass}, we have $\|u(t)\|_1=\|u_0\|_1$ which implies \eqref{eq_lp_decay} for $p=1$. Hence, we can assume that $p>1$. 

We multiply equation \eqref{aggr} by $pu^{p-1}$ (recall that $u$ is nonnegative), integrate with respect to $x$ over $\R$, and integrate by parts to obtain

\begin{align*}
\frac{d}{dt}\int_\mathbb{R}u^p\dx=-\frac{4(p-1)\varepsilon}{p}\int_\mathbb{R} \left[\left(u^{p/2}\right)_x\right]^2\dx + (p-1)\int_\mathbb{R}u^p K'*u_x\dx.
\end{align*}
First term on the right-hand side (containing $\varepsilon>0$) is obviously nonpositive, hence, we skip it in our estimates. Using the explicit form of the kernel $K'=-\frac{A}{2}H+V$ and Lemma \ref{conv_ker}, we rewrite the second term as follows:
\begin{align}\label{est1}
(p-1)\int_\mathbb{R}u^p K'*u_x\dx=(p-1)\left(-A\int_\mathbb{R}u^{p+1}\dx+\int_\mathbb{R} u^p\ V_x*u\dx\right).
\end{align}
Notice, that a simple computation involving the H\" older and the Young inequalities leads to the estimates
\begin{equation}\label{est2}
\left|\int_\mathbb{R} u^p\ V_x*u\dx\right|\leq ||V_x*u||_{p+1}||u^p||_\frac{p+1}{p}
\leq ||V_x||_1||u||_{p+1}^{p+1}.
\end{equation}
Hence, using \eqref{est1} and \eqref{est2} we get
\begin{align}\label{pr2.1}
\frac{d}{dt}\int_\mathbb{R}u(x,t)^p\dx\leq (p-1)\left(-A+||V_x||_1\right)||u(t)||_{p+1}^{p+1}.
\end{align}

Moreover, it follows from the H\" older inequality (with the exponents $p$ and $\frac{p}{p-1}$) that
\begin{equation*}
\int_\mathbb{R}u^p\dx=\int_\mathbb{R}u^{\frac{1}{p}}\ u^{\frac{p^2-1}{p}}\dx\leq \Big{(}\int_\mathbb{R}u\dx\Big{)}^{\frac{1}{p}}\Big{(}\int_\mathbb{R}u^{p+1}\dx\Big{)}^\frac{p-1}{p},
\end{equation*}
which means
\begin{equation}\label{est3}
\int_\mathbb{R}u^{p+1}\dx\geq ||u_0||_1^\frac{-1}{p-1} \Big{(}\int_\mathbb{R}u^p\dx\Big{)}^\frac{p}{p-1},
\end{equation}
because $\|u(t)\|_1=\|u_0\|_1$.
Applying estimate \eqref{est3} to \eqref{pr2.1}, we obtain the following differential inequality for $\int_\mathbb{R}u^p\dx$:
\begin{align}\label{pr2.2}
\frac{d}{dt}\int_\mathbb{R}u(x,t)^p\dx\leq (p-1)\left(-A+||V_x||_1\right)||u_0||_1^{-\frac{1}{p-1}} \left(\int_\mathbb{R}u(x,t)^p\dx\right)^\frac{p}{p-1}.
\end{align}
It is easy to prove that any nonnegative solution of the differential inequality
\begin{align*}
\frac{d}{dt}f(t)\leq -Df(t)^{\frac{p}{p-1}},
\end{align*}
with a constant $D>0$, satisfies
\begin{align*}
f(t)\leq\left(\frac{D}{p-1}\right)^{1-p}t^{1-p}.
\end{align*}
Hence, it follows from \eqref{pr2.2} and from the assumption $\|V_x\|_1<A$ that
\begin{align}\label{pr2.3}
||u(t)||_p\leq \left(A-||V_x||_1\right)^\frac{1-p}{p} ||u_0||_1^{1/p}\ t^\frac{1-p}{p}
\end{align}
for all $t>0$.
Finally, passing to the limit $p\to\infty$ in \eqref{pr2.3} we obtain
\begin{align*}
||u(t)||_\infty\leq \left(A-||V_x||_1\right)^{-1}\ t^{-1}
\end{align*}
for all $t>0$. This completes the proof of Theorem \ref{thm_lp_decay}.
\end{proof}

Let us now recall some result on smooth approximations of rarefaction waves, more precisely, the solution of the following Cauchy problem:

\begin{equation}\label{smooth_raref}
\begin{array}{l}
Z_t-\varepsilon Z_{xx}+AZZ_x=0,\\
Z(x,0)=Z_0(x)=\frac{1}{2}H(x).
\end{array}
\end{equation}
where $A>0$.

\begin{lemm}[Hattori-Nishihara \cite{HN91}]\label{HT}
Problem \eqref{smooth_raref} has a unique, smooth, global-in-time solution $Z(x,t)$ satisfying
\begin{enumerate}[i)]
\item $-1/2<Z(x,t)<1/2$ and $Z_x(x,t)>0$ for all $(x,t)\in\R\times(0,\infty)$;
\item for every $p\in[1,\infty]$, there exists a constant $C=C(p)>0$ independent of $\varepsilon>0$ such that
\begin{equation*}
\|Z_x(t)\|_p\leq Ct^{-1+1/p}
\end{equation*}
and
\begin{equation*}
\|Z(t)-W^R(t)\|_p\leq Ct^{-(1-1/p)/2}
\end{equation*} \label{ineq}
\end{enumerate}
for all $t>0$, where $W^R(x,t)$ is the rarefaction wave given by formula \eqref{raref}.
\end{lemm}

\par\noindent{\sc Sketch of the proof.} All results stated in Lemma \ref{HT} can be found in \cite{HN91} with some additional improvements contained in \cite[sect. 3]{KT04}, and they are deduced from an explicit formula for smooth approximation of rarefaction waves. Here however, we should emphasise that the authors of \cite{HN91} consider equation \eqref{smooth_raref} with $\varepsilon=1$ but, by a simple scaling argument, we can extend those results for all $\varepsilon>0$. Indeed, we check that the function $f(x,t)=Z(\varepsilon x,\varepsilon t)$ satisfies $f_t-f_{xx}+Aff_x=0$.
Hence, by the result from \cite{HN91} we have
\begin{align*} 
\|f_x(t)\|_p\leq Ct^{\frac{1-p}{p}}\qquad\text{and}\qquad \|f(t)-W^R(t)\|_p\leq Ct^{-(1-1/p)/2}.
\end{align*}
Now, coming back to original variables, we have
\begin{align*}
\varepsilon^{\frac{p-1}{p}}\|Z_x(\cdot,\varepsilon t)\|_p\leq C\ (\varepsilon t)^{\frac{1-p}{p}} \varepsilon^{\frac{p-1}{p}}
\end{align*}
and so, defining the new variable $\tilde{t}=\varepsilon t$, we obtain $\|Z_x(\tilde{t})\|_p\leq C\ \tilde{t}^{\frac{1-p}{p}}$
with a constant $C$ independent of $\varepsilon$. A similar reasoning should be applied in the case of the second inequality in Lemma \ref{HT}.\ref{ineq}.
\cbdu

Next, we study the large time asymptotics of $U(x,t)=\int_{-\infty}^xu(y,t)\dy{y}-\frac{1}{2}$, which satisfy equation \eqref{primit}. Recall that $u=U_x$. In the proof of Theorem \ref{conv}, we need the following auxiliary result.

\begin{lemm}\label{bound_h}
Let $u_0$ satisfy conditions \eqref{as-u0}. Assume that $U=U(x,t)$, defined by \eqref{primit}, is the solution of equation \eqref{primitive} supplemented with the initial condition $U_0(x)=\int_{-\infty}^xu_0(y)\dy{y}-1/2$ and $Z=Z(x,t)$ is the smooth approximation of the rarefaction wave, namely, the solution of problem \eqref{smooth_raref}. Then, for every $t_0>0$ we have
\begin{equation*}
\sup_{t>t_0}\frac{1}{\log(2+t)}\|U(t)-Z(t)\|_1<\infty
\end{equation*}
\end{lemm}
\dowod
At the beginning, let us notice that assumption \eqref{as-u0} on $u_0$ imply that $U_0(x)\in L^1(-\infty,0)$ and $U_0(x)-1\in L^1(0,\infty)$. Hence, we have that $U_0-Z_0\in L^1(\R)$.

Denoting $R=U-Z$ and using equations \eqref{primitive} and \eqref{smooth_raref}, we see that this new function satisfies
\begin{equation*}
R_t=\varepsilon R_{xx}-\frac{A}{2}(U^2-Z^2)_x+U_x\ V*U_x.
\end{equation*}
We multiply this equation by $\sgn R$ (in fact, by a smooth approximation of $\sgn R$) and we integrate with respect to $x$ to obtain
\begin{align*}\label{h_lp}
\frac{d}{dt}\int_\R |R|\dx=\varepsilon\int_\R R_{xx}\sgn R\dx -\frac{A}{2}\int_\R (U^2-Z^2)_x\sgn R\dx+ 
\int_\R U_x\ V*U_x\sgn R\dx.
\end{align*}

The first term on the right-hand side of the above equation is nonpositive because this is the well-known Kato inequality. 
The second term is equal to $0$ because of the following calculations:
\begin{align*}
\int_\R (U^2-Z^2)_x\sgn R\dx&=\int_\R \left(R^2+2RZ\right)_x\sgn R\dx\\
 &=\int_\R 2R_x|R|\dx+\int_\R 2ZR_x\sgn R\dx+ \int_\R 2Z_x|R|\dx\\
&=-2\int_\R Z_x|R|\dx+2\int_\R Z_x|R|\dx=0
\end{align*}
since $\int_\R R_x|R|\dx=0$. Moreover, using the Young inequality, we have
\begin{align*}
\left|\int_\R U_x\ V*U_x\ \sgn R\dx\right|\leq \|U_x\ V*U_x\|_1\leq \|U_x\|_\infty\|V\|_1\|U_x\|_1.
\end{align*}
Hence, by the fact that $U_x(t)=u(t)$ and using the decay estimates from Theorem \ref{thm_lp_decay} for $p=1$ and $p=\infty$ we get the following differential inequality
\begin{equation*}
\frac{d}{dt}\|R(t)\|_1\leq Ct^{-1}
\end{equation*}
which completes the proof of Lemma \ref{bound_h}.
\cbdu

Now, we are in a position to prove our main result about convergence the primitive of $u$ towards a rarefaction wave.

\begin{proof}[Proof of Theorem \ref{conv}]
Let $Z=Z(x,t)$ be the smooth approximation of the rarefaction wave from Lemma \ref{HT}. Denote $R=Z-U$. Hence, by Lemma \ref{HT} and Theorem \ref{thm_lp_decay}, we have
\begin{equation*}
\|R_x(t)\|_\infty=\|U_x(t)-Z_x(t)\|_\infty\leq\|u(t)\|_\infty+\|Z_x(t)\|_\infty\leq C\ t^{-1}
\end{equation*}
for a constant $C>0$.
Moreover, using the Sobolev-Gagliardo-Nirenberg inequality
\begin{equation*}
\|R\|_p\leq C\|R_x\|^{\frac{1}{2}\left(1-\frac{1}{p}\right)}_\infty \|R\|_1^{\frac{1}{2}(1+\frac{1}{p})},
\end{equation*}
valid for every $p\in(1,\infty]$ and Lemma \ref{bound_h} we have
\begin{equation*}
\|U(t)-Z(t)\|_p\leq Ct^{-\frac{1}{2}\left(1-\frac{1}{p}\right)}\left(\log(2+t)\right)^{\frac{1}{2}(1+\frac{1}{p})}
\end{equation*}
for all $t>0$.

Finally, to complete the proof, we use Lemma \ref{HT} to replace the smooth approximation $Z(x,t)$ by the rarefaction wave $W^R(x,t)$.
\end{proof}

The proof of Theorem \ref{thm_eps} relies on a form of Aubin-Simon's compactness result that we recall below.

\begin{theo}[{\cite[Theorem 5]{Simon}}]\label{simon}
Let $X$, $B$ and $Y$ be Banach spaces satisfying $X \subset B \subset Y$
 with compact embedding $X \subset B$. Assume, for $1 \le p \le
 +\infty$ and $T > 0$, that 
\begin{itemize}
\item $F$  is bounded in  $L^p (0, T; X)$, 
\item $\{\partial_t f\,:\, f\in F\}$  is bounded in  $L^p (0, T; Y)$. 
\end{itemize}
Then $F$ is relatively compact in $L^p (0, T; B)$ {\rm (}and in $C (0, T; B)$
 if $p = +\infty${\rm )}.
\end{theo}

\begin{proof}[Proof of Theorem \ref{thm_eps}]
We denote $U^\varepsilon$ as a solution of equation \eqref{primitive} with $\varepsilon>0$ supplemented with a initial condition \eqref{U_ini}.
The proof follows three steps: first we show that the family 
\begin{align*}
{\mathcal F}\equiv\{ U^\varepsilon: \varepsilon\in(0,1]\},
\end{align*}
is relative compact in $C([t_1,t_2],C[-R,R])$  for every $0<t_1<t_2<\infty$ and every $R>0$. Next, we show that there exist a function $\bar{U}=\lim_{\varepsilon\to 0}U^\varepsilon$ which is a weak solution of problem \eqref{eq_U}-\eqref{U_ini}. Finally we prove that $\bar{U}$ satisfy estimate \eqref{thm_eps_WR}.

{\it Step 1. Compactness}.
We apply Theorem \ref{simon} with $p=\infty$, $F = {\mathcal F}$, and
\[
 X = C^1([-R, R]), \qquad B = C ([-R, R]), \qquad  Y
 = W^{-1, 1}([-R, R]),
\] 
where $R > 0$ is fixed and arbitrary, and $Y$ is the dual space of $W^{1, 1}_0 ([-R, R])$. %
Obviously,  the embedding $X \subseteq B$ is compact by the Arzela-Ascoli theorem.

First, we show that the sets ${\mathcal F}$ and $ \{\partial_x U^\varepsilon: \varepsilon\in(0,1]\}$ are bounded subsets of $L^\infty \left([t_1 , t_2], C([-R, R])\right)$. Indeed, it follows from definition of function $U^\varepsilon$, namely from \eqref{primit}, that
\begin{align}\label{proof_eq1}
\left| U^\varepsilon(x,t)\right|\leq \|(U^\varepsilon)_x(\cdot, t)\|_1+\frac{1}{2}=\|u_0\|_1+\frac{1}{2}.
\end{align}
Moreover, using Theorem \ref{thm_lp_decay} we have
\begin{align}\label{bound_u}
\|(U^\varepsilon)_x(\cdot,t)\|_\infty\leq (A-\|V_x\|_1)^{-1}\ t^{-1}.
\end{align}

To check the second condition of Aubin-Simon's compactness criterion, it is suffices to show that there is a
 positive constant $C$ which independent of $\varepsilon\in(0,1]$ such that 
$
 \sup_{t \in [t_1 , t_2]} \|\partial_t U^\varepsilon \|_Y \leq C.
$
Let us show this estimate by a duality argument. 
For every $\varphi \in C^\infty_c \left((-R,
 R)\right)$ and  $t \in [t_1 , t_2]$, by \eqref{proof_eq1}, \eqref{bound_u} and Theorem \ref{thm_lp_decay}, we have 
\begin{align*}
&\left|\int_\R \partial_t U^\varepsilon (t) \varphi\, \dx \right| \leq
 \left|\int_\R \varepsilon U^\varepsilon_x(t) \varphi_x\dx\right| + \left|\int_\R AU^\varepsilon(t) U^\varepsilon_x(t)\varphi\dx\right| +
\left|\int_\R U^\varepsilon_x(t) V*U^\varepsilon_x(t)\varphi\dx\right| \\
&\leq \|\varphi_x\|_\infty \int_\R |U_x^\varepsilon(t)|\dx + 
 A\|U^\varepsilon(t)\|_\infty \|\varphi\|_\infty\int_\R \left|U^\varepsilon_x(t)\right|\dx+\|U^\varepsilon_x(t)\|_\infty^2\|V\|_1 \|\varphi\|_1 \\
&\leq \|\varphi_x\|_\infty\|u_0\|_1+ A\|u_0\|_1(\|u_0\|_1+1/2)\|\varphi\|_\infty+ (A-\|V_x\|_1)^{-2}t_1^{-2}\|V\|_1 \|\varphi\|_1.
\end{align*}
Hence, the proof of Step 1 is completed.

{\it Step 2. Limit function.}
By Step 1, for every  $0 < t_1 < t_2 < +\infty$, 
the family $\{ U^\varepsilon: \varepsilon\in(0,1]\}$ is
 relatively compact in $C([t_1 , t_2], C(-R,R))$. Consequently, by a diagonal argument, there exists a sequence of
 $\{ U^{\varepsilon_n}: \varepsilon_n\in(0,1]\}$ and a function $\bar{U}
 \in C((0, +\infty), C(\R))$ such that 
\begin{align}\label{lim_U}
U^{\varepsilon_n}\to \bar{U} \quad \text{as}\quad \varepsilon_n \to 0\quad \text{in}\quad L^\infty_{loc}\big{(}\R \times (0, +\infty)\big{)}.
\end{align}
Moreover, by the Banach-Alaoglu Theorem, it follows from the estimate \eqref{bound_u} that
\begin{align*}
U^{\varepsilon_n}_x\to\bar{U}_x  \quad\text{as}\quad \varepsilon_n\to 0
\end{align*}
weak-$*$ in $L^\infty_{loc}\big{(}(0,\infty),L^\infty(\R)\big{)}$.

Now, multiplying equation \eqref{primitive} by a test function $\varphi \in
 C^\infty_c (\R \times [0, +\infty))$ and integrating the resulting equation over $\R\times [0,\infty)$, we obtain the identity
\begin{multline}\label{weak_lim}
-\int_0^\infty\int_\R U^{\varepsilon_n}\varphi_t\dx\dy{t} - \int_\R U_0(x)\varphi(x,0)\dx = \varepsilon_n\int_0^\infty\int_\R U^{\varepsilon_n}\varphi_{xx}\dx\dy{t}
\\ +\frac{A}{2}\int_0^\infty\int_\R (U^{\varepsilon_n})^2\varphi_x\dx\dy{t} + \int_0^\infty\int_\R U^{\varepsilon_n}_x\,\big{(}V_x*U^{\varepsilon_n}\big{)}\,\varphi\dx\dy{t}
\end{multline}

It is easy to pass to the limit $\varepsilon_n\to 0$ in left-hand side of \eqref{weak_lim}, using the Lebesgue dominated convergence theorem. To deal with term in the right-hand side we make the following decomposition:
\begin{align}\label{thm_eps_eq1}
\int_\R U^{\varepsilon_n}_x\left( V_x*U^{\varepsilon_n}\right)\varphi\dx= \int_\R U^{\varepsilon_n}_x\left( V_x*(U^{\varepsilon_n}-\bar{U})\right)\varphi\dx+
\int_\R U^{\varepsilon_n}_x\left( V_x*\bar{U}\right)\varphi\dx.
\end{align}
We can estimate the first term on the right-hand side of \eqref{thm_eps_eq1} as follows:
\begin{align}\label{thm_eps_eq2}
\left|\int_\R U^{\varepsilon_n}_x\left( 
V_x*(U^{\varepsilon_n}-\bar{U})\right)\varphi\dx\right|\leq \|U^{\varepsilon_n}_x(t)\|_\infty \int_\R \left|V_x*(U^{\varepsilon_n}-\bar{U})\varphi\right|\dx
\end{align}
Let us notice, that $V_x*(U^{\varepsilon_n}-\bar{U})$ tends to zero as $\varepsilon_n\to 0$ by Lebesgue dominated convergence theorem and it is bounded independently of $\varepsilon_n$. Hence, using the Lebesgue dominated convergence theorem and Theorem \ref{thm_lp_decay}, we deduce that the right-hand side of \eqref{thm_eps_eq2} converge to zero. The second term on the right-hand side of \eqref{thm_eps_eq1} obviously converge to $\int_\R \bar{U}_x\left( V_x*\bar{U}\right)\varphi\dx$ by the weakly-$*$ convergence of $U^{\varepsilon_n}_x$ in $L^\infty(\R)$ since $(V_x*\bar{U})\varphi\in L^1(\R)$. This completes the proof of Step 2.

{\it Step 3. Convergence towards rarefaction wave.}
To prove \eqref{thm_eps_WR}, we use the Fatou Lemma and \eqref{lim_U}, to obtain
\begin{align*}
\|\bar{U}(t)-W^R(t)\|_p\leq \liminf_{\varepsilon_n\to 0} \|U^{\varepsilon_n}(t)-W^R(t)\|_p
\end{align*}
for all $t>0$.

Now, it is enough to use Theorem \ref{conv} to estimate the quantity on right-hand side, since constant $C$ in \eqref{conv_raref} is independent of $\varepsilon$. Hence the proof of Theorem \ref{thm_eps} is finished.
\end{proof}

At last, we prove Corollary \ref{corolary}.

\begin{proof}[Proof of Corollary \ref{corolary}]
First, we express the result stated in Theorems \ref{conv} and \ref{thm_eps} in another way. We consider the rescaled family of function
$
U^\lambda (x,t)=U(\lambda x,\lambda t)
$
for all $\lambda>0$. Let us also notice that $W^R(x,t)$ is self-similar in the sense that $\left(W^R\right)^\lambda(x,t)=W^R(x,t)$ for all $x\in\R$, $t>0$, $\lambda>0$. Hence, changing the variables and using Theorem~\ref{conv} and Theorem \ref{thm_eps} for the case $\varepsilon=0$, we obtain
\begin{multline*}
\|U^\lambda(\cdot,t_0)-\left(W^R\right)^\lambda(\cdot,t_0)\|_p=\lambda^{-1/p}\|U(\cdot,\lambda t_0)-W^R(\cdot,\lambda t_0)\|_p\leq\\
C\lambda^{-1/p}(\lambda t_0)^{-\frac{1}{2}\left(1-\frac{1}{p}\right)}\left( \log(2+\lambda t_0)\right)^{\frac{1}{2}\left(1+\frac{1}{p}\right)}\to 0
\end{multline*}
as $\lambda\to\infty$. It means that the family of functions $U^\lambda$ converge in $L^p(\R)$ as $\lambda\to\infty$ towards $W^R(x,t)$ for every $t_0>0$ and $p\in(1,\infty]$.

This scaling argument allows us to express the convergence of solutions to original problem \eqref{aggr}-\eqref{aggr-ini} towards a self-similar profile.
Indeed, let us note that since $u=U_x$, it follows immediately that $u^\lambda(x,t)=\lambda u(\lambda x,\lambda t)=\partial_x U^\lambda(x,t)$. Hence, the weak convergence of $u^\lambda$ towards the distributional derivative of the rarefaction wave $\partial_x W^R$ is the immediate consequence of the Lebesgue dominated convergence theorem and of Theorem \ref{conv} for $p=\infty$ since $|U^\lambda(x,t_0)|\leq \int_\R u_0(x)\dx+\frac{1}{2}$
\end{proof}

\end{document}